\documentclass{article}
\usepackage{amssymb,amsmath,amsthm,latexsym, enumerate}
\usepackage{latexsym}
\newtheorem{theorem}{Theorem}[section]
\newtheorem{proposition}{Proposition}[section]
\newtheorem{definition}{Definition}[section]
\newtheorem{corollary}{Corollary}[section]
\newtheorem{lemma}{Lemma}[section]

\newcommand{\rad}{\operatorname{rad}}

\begin{document}

\title{Cell algebra stuctures on monoid and twisted monoid algebras}

\author{ Robert May}
\date{}
\maketitle

\section {Introduction}

In \cite{May} a class of algebras called cell algebras is defined generalizing the cellular algebras of Graham and Lehrer in \cite{GL}.  (These algebras had previously been introduced and studied as ``standardly based algebras'' by Du and Rui in \cite{DR}.)  It was shown in \cite{May} that if $R$ is any commutative domain with unit and $M$ is either a transformation semigroup $\mathcal{T}_r $ or a partial transformation semigroup $\mathcal{PT}_r $ then the semigroup algebra $R[M]$ is a cell algebra with a standard cell basis.  In this paper we show that the algebra $R[M]$ is a cell algebra for any finite monoid $M$ with the property (called the ``$R$-C.A.'' property) that for every $\mathcal{D}$-class $D$ the group algebra $R[G_D ]$ for the Schutzenberger group of $D$ has a cell algebra structure.  The choice of a cell basis for each $R[G_D ]$ determines a standard cell basis for $R[M]$.  If all the groups $G_D$ for a given $M$ are symmetric groups (as for $\mathcal{T}_r $ or $\mathcal{PT}_r $) then $M$ satisfies the $R$-C.A. property for any $R$.  If $k$ is an algebraically closed field of characteristic $0$ or $p$ where $p$ does not divide the order of any $G_D$,  then any finite monoid $M$ satisfies the $k$-C.A. property. 

We use the standard cell basis obtained for $M$ and the properties of cell algebras given in \cite{May} to derive properties of the algebras $R[M]$ for such monoids.  For example, if $M$ is a regular semigroup and each cell algebra $R[G_D ]$ has the property that $\left( {\Lambda _D } \right)_0  = \Lambda _D $ then $R[M]$ is quasi-hereditary.  As a special case, this yields the theorem of Putcha \cite{Pu} that the complex monoid algebra $\mathbb{C}[M]$ is quasi-hereditary for any finite regular monoid.  If $M$ is an inverse semigroup satisfying the $R$-C.A. property, then we show that $R[M]$ is semi-simple if and only if each cell algebra $R[G_D ]$ is semi-simple.  When $R$ is an algebraically closed field $k$ of good characteristic, this yields that $k[M]$ is semi-simple for any inverse semi-group (a result long known even without the algebraically closed condition; see \cite{CP} which cites \cite{Og}).

We also show that if $\pi :M \times M \to R$ is any twisting satisfying a certain simple compatibility condition then the twisted monoid algebra $R^\pi  [M]$ is also a cell algebra (with the same sets $\Lambda ,L,R$
 and cell basis $C$ as the cell algebra $R[M]$).

In \cite{East}, \cite{Wil}, and \cite{GX}, East, Wilcox, and Guo and Xi have explored analogous results for cellular algebra structures on semigroups and twisted semigroups.  Many of the complications in their work arise from the need to construct the involution anti-automorphism $ * $ required by the definition of a cellular algebra.  Cell algebras require no such mapping, and the results obtained for such algebras appear to be both considerably simpler and of more general applicability than the corresponding result for cellular algebras.  Yet questions such as whether an algebra is quasi-hereditary or semi-simple appear to be not much harder to address when given a cell basis than when given a cellular basis.

\section  {Properties of finite monoids}

In this section we review some facts about finite monoids and the standard Green's relations on a finite monoid $M$.  Most of these results are well-known; see, for example, \cite{CP}, \cite{GM}, or \cite{How} for references.

Green's relations $\mathcal{L ,\: R, \: J, \: H}$ and $\mathcal{D}$  are defined, for $x,y \in M$, by
\[x\mathcal{L}y \Leftrightarrow Mx = My\:,\,x\mathcal{R}y \Leftrightarrow xM = yM\:,\,x\mathcal{J}y \Leftrightarrow MxM = MyM\]
\[x\mathcal{H}y \Leftrightarrow x\mathcal{L}y{\text{ and }}x\mathcal{R}y \]
\[x\mathcal{D}y \Leftrightarrow x\mathcal{L}z{\text{ and }}z\mathcal{R}y{\text{ for some }}z \in M \Leftrightarrow x\mathcal{R}z{\text{ and }}z\mathcal{L}y{\text{ for some }}z \in M.\]
 If $\mathcal{K}$ is one of Green's relations and $m \in M$, denote by $K_m $ the $\mathcal{K}$ -class containing $m$. Write $\mathbb{L,\: R,\: J,\: H,}$ or $\mathbb{D}$ for the set of all $\mathcal{L,\: R,\: J,\: H,}$ or $\mathcal{D}$ classes in $M$.
 
According to ``Green's lemma'', if $a \mathcal{R} b $ and $s,t \in M $ are such that $as=b \, , \,bt=a $, then the map $ x \to xs  $ is a bijective, $\mathcal{R}$-class preserving map from $L_a$ onto $L_b$ with inverse the map $ y \to yt $.  There is a dual result when $a \mathcal{L} b $

For a finite monoid $M$, one has $\mathcal{D} = \mathcal{J}$ (see chapter 5 of \cite{Gr}) and the set $\mathbb{D}$ of $\mathcal{D}$-classes inherits a partial order defined, for $x,y \in M$, by $D_x  \leqslant D_y  \Leftrightarrow x \in MyM$.  Evidently $D_{xy}  \leqslant D_x $ and $D_{xy}  \leqslant D_y $ for any $x,y \in M$.

Let $R$ be a commutative domain with unit.  For any subset $S \subseteq M$, write $R[S]$ for the free $R$-module with basis $S$.  Let $A = R[M]$, the monoid algebra for $M$, and define $A^D  = \mathop  \oplus \nolimits_{D' \leqslant D} R[D']\,,\,\hat A^D  = \mathop  \oplus \nolimits_{D' < D} R[D']$.  Then both $A^D $ and $\hat A^D $ are two sided ideals in $A$. 

For a given $D \in \mathbb{D}$, choose a ``base element'' $\gamma  \in D$
 and consider the $\mathcal{H}$-class $H = H_\gamma   \subseteq D$.  Let $RT(H) = \left\{ {m \in M:Hm \subseteq H} \right\}$ and for $m \in RT(H)$ define the right translation $r_m :H \to H$ by $(h)r_m  = hm\,,\,h \in H$.  Each map $r_m $ is a bijection from $H$ to $H$ and $G_H^R  = \left\{ {r_m :m \in RT(H)} \right\}$ is a group, the (right) Schutzenberger group for $H$.  (Up to isomorphism, $G_H^R $ depends only on $D$, not the particular $\mathcal{H}$-class $H$ contained in $D$.)  Evidently $r_m r_n  = r_{mn} $ for any $m,n \in RT(H)$.  The map $\phi _H^R :G_H^R  \to H$ defined by $\phi _H^R :r_m  \mapsto \gamma _D m$ is bijective and extends linearly to a bijective map $\phi _H^R :R\left[ {G_H^R } \right] \to R\left[ H \right]$ of free $R$-modules.  For any $m,n \in RT(H)$, $\phi _H^R (r_m ) \cdot n = (\gamma _D m) \cdot n = \gamma _D (mn) = \phi _H^R \left( {r_{mn} } \right) = \phi _H^R \left( {r_m r_n } \right)$, so in general for any $x \in R\left[ {G_H^R } \right]$ we have by linearity $\phi _H^R \left( x \right) \cdot n = \phi _H^R \left( {xr_n } \right)$. 

In parallel fashion, let $LT(H) = \left\{ {m \in M:mH \subseteq H} \right\}$ and for $m \in LT(H)$ define the left translation $l_m :H \to H$ by $l_m (h) = mh\,,\,h \in H$.  Then again each map $l_m $ is a bijection from $H$ to $H$ and $G_H^L  = \left\{ {l_m :m \in LT(H)} \right\}$ is a group, the (left, or dual) Schutzenberger group for $H$.  
Then $l_m l_n  = l_{mn} $ for any $m,n \in LT(H)$, the map $\phi _H^L :G_H^L  \to H$ defined by $\phi _H^L :l_m  \mapsto m\gamma _D $ is bijective, and the linearly extended map $\phi _H^L :R\left[ {G_H^L } \right] \to R\left[ H \right]$ is a bijection of free $R$-modules.  The map $\psi  = \left( {\phi _H^R } \right)^{ - 1}  \circ \phi _H^L :G_H^L  \to G_H^R $ is then also bijective.  Choose an (injective) map $\vartheta :G_H^R  \to RT(H)$ such that $r_{\vartheta (g)}  = g$ for all $g \in G_H^R $.  Then for any $m \in LT(H)$ define $\bar m \in RT(H)$ by $\bar m = \vartheta  \circ \psi \left( {l_m } \right)$, so that $m\gamma _D  = \gamma _D \bar m$.  For any $m \in LT(H)\,,\,n \in RT(H)$ we have $m \cdot \phi _H^R \left( {r_n } \right) = m \cdot \gamma _D n = m\gamma _D  \cdot n = \gamma _D \bar m \cdot n = \phi _H^R \left( {r_{\bar m} r_n } \right)$.  Then for any $x \in R\left[ {G_H^R } \right]\,,\,m \in LT(H)$
 we have by linearity $m\phi _H^R \left( x \right) = \phi _H^R \left( {r_{\bar m} x} \right)$.
  
\begin{lemma} \label{l2.1}
  Take any $m \in M$.
\begin{itemize}
\item [i.]  If there exists an $\,h \in H$ such that $hm \in H$ then $m \in RT(H)$ and $r_m $ gives a bijective map of $H$ onto $H$.
\item [ii.] If there exists an $\,h \in H$ such that $mh \in H$then $m \in LT(H)$ and $l_m $ gives a bijective map of $H$ onto $H$.
\end{itemize}
\end{lemma}  

\begin{proof}

For i., write $h = m_1  \cdot \gamma _D \,,\,hm = m_2  \cdot \gamma _D $ for some $m_1 ,m_2  \in LT(H)$.  Then $l_{m_1 }  \in G_H^L $ and we can choose $\bar m_1  \in LT(H)$
 such that $l_{\bar m_1 }  = \left( {l_{m_1 } } \right)^{ - 1} $.  Then $\gamma _D  = \bar m_1  \cdot h$.  Any element of $H$ has the form $n \cdot \gamma _D $ for some $n \in LT(H)$, and we have $n\gamma _D  \cdot m = n \cdot \left( {\bar m_1 h} \right) \cdot m = n\bar m_1 (hm) = n\bar m_1 (m_2 \gamma _D ) = \left( {n\bar m_1 m_2 } \right) \cdot \gamma _D $.  But $n\bar m_1 m_2  \in LT(H)$ (since $n,\bar m_1 ,m_2 $ are), so $n\gamma _D  \cdot m = \left( {n\bar m_1 m_2 } \right) \cdot \gamma _D  \in H$.  Then $m \in RT(H)$ as desired.  The proof of ii. is parallel.  
\end{proof}

We now investigate general products $am\,,\,ma$ where $a \in D \in \mathbb{D}$ and $m \in M$.  As mentioned above, we have $D_{am}  \leqslant D_a $ and $D_{ma}  \leqslant D_a $.  We need to study the cases when $D_{am}  = D_a $  or $D_{ma}  = D_a $.  The first tool is the following result (see e.g chapter 5 of \cite{Gr}):

\begin{lemma}  \label{l2.2}
  For any finite monoid $M$ and $a,m \in M$, 
\begin{itemize}
\item [i] If $D_a  = D_{am} $, then $R_a  = R_{am} $.
\item [ii]  If $D_a  = D_{ma} $, then $L_a  = L_{ma} $.
\end{itemize}
\end{lemma}

\begin{proof}
 We use the facts (again, see chapter 5 of \cite{Gr}) that in any finite monoid $\mathcal{D} = \mathcal{J}$ and that for any element $x$ in a finite monoid some power $x^n $ is idempotent.
For i., assume $D_a  = D_{am} $ for some $a,m \in M$.  To prove $R_a  = R_{am} $, it suffices to check that $a \in amM$ and $am \in aM$.  Since $am \in aM$ is obvious, we need show only $a \in amM$.  Now $D_a  = D_{am}  \Rightarrow J_a  = J_{am}  \Rightarrow a \in MamM \Rightarrow a = x(am)y = xa(my)$ for some $x,y \in M$.  Then by repeated substitutions we have $a = x^k a(my)^k $ for any positive integer $k$.  Choose $k=n$ so that $x^n $ is an idempotent $e$.  Then $a = ea(my)^n $ and $ea = e \cdot ea(my)^n  = ea(my)^n  = a$.  But then $a = a(my)^n  = am \cdot y(my)^{n - 1}  \in amM$ as desired.  The proof of ii. is similar. 
\end{proof}

We also use the ``egg-box'' picture of a given fixed class $D \in \mathbb{D}$:

Picture the $\mathcal{H}$-classes contained in $D$ in a rectangular array where the $\mathcal{H}$-classes in a given column are all in the same $\mathcal{L}$-class and the $\mathcal{H}$-classes in a given row are all in the same $\mathcal{R}$-class.  So the number of columns in the array is the number $n(D,L)$ of distinct $\mathcal{L}$-classes in $D$, while the number of rows is the number $n(D,R)$ of distinct $\mathcal{R}$-classes in $D$.  Write ${}_iH_j $ for the $\mathcal{H}$-class in row $i$, column $j$.  We can assume that ${}_1H_1 $ is our ``base class'' $H_\gamma  $.

The following result follows from Green's lemma and its dual:

\begin{lemma} \label{l2.3}
  For each $i \in \left\{ {1,2, \cdots ,n(D,R)} \right\}\,,\,j \in \left\{ {1,2, \cdots ,n(D,L)} \right\}$ there exist elements $a_i ,\bar a_i ,b_j ,\bar b_j  \in M$ such that
\begin{itemize}
\item [i.]  For each column $j$, the left translation $h \mapsto a_i h$ defines a bijective map $l_i :{}_1H_j  \to {}_iH_j $ with inverse $l_i ^{ - 1} :{}_iH_j  \to {}_1H_j $ defined by left translation $l_i ^{ - 1} :h \mapsto \bar a_i h$.
\item [ii.]  For each row $i$, the right translation $h \mapsto hb_j $ defines a bijective map $r_j :{}_iH_1  \to {}_iH_j $ with inverse $r_j ^{ - 1} :{}_iH_j  \to {}_iH_1 $ defined by right translation $r_j ^{ - 1} :h \mapsto h\bar b_j $.
\end{itemize}
\end{lemma}

Then for any row $i$, column $j$, and $h \in H_\gamma  $,  $h \mapsto a_i hb_j $ gives a bijective map ${}_i\psi _j :H_\gamma   = {}_1H_1  \to {}_iH_j $ of $\mathcal{H}$-classes which extends linearly to a bijection ${}_i\psi _j :R[H_\gamma  ] \to R[{}_iH_j ]$.

 We can now state the main result of this section.

\begin{proposition}  \label{p2.1}
  Suppose that for some $d \in D,m \in M$ we have $dm \in D$.   For some $i$,$j$, this $d \in {}_iH_j  = {}_i\psi _j \left( {H_\gamma  } \right) = a_i H_\gamma  b_j $ and so $d = a_i \phi _H^R (g_d )b_j $ for some $g_d  \in G_H^R $.  Then:

\begin{itemize}
\item [i.]  $dm \in {}_iH_k $ for some $k$.  Define $m^ *   = b_j m\bar b_k $.
\item [ii.]  $m^*  \equiv b_j m\bar b_k  \in RT\left( H \right)$.
\item [iii.]  The right translation $r_m :h \mapsto hm$ gives a bijection $r_m :{}_iH_j  \to {}_iH_k $.
\item [iv.]  If $h = a_i \phi _H^R (g)b_j  \in {}_iH_j $, then $hm = (h)r_m = a_i \phi _H^R \left( {g\,r_{m^ *  } } \right)b_k $.
\item [v.]  For any $x \in R\left[ {{}_iH_j } \right]$, write $x = a_i \phi _H^R \left( y \right)b_j $ for some $y \in R\left[ {G_H^R } \right]$.  Then $xm = a_i \phi _H^R \left( {y\,r_{m^ *  } } \right)b_k $.
\end{itemize}
\end{proposition}

\begin{proof}
  Since $D_d  = D_{dm}  = D$, we have $R_d  = R_{dm} $ by lemma 2.2.  So $H_d {\text{ and }}H_{dm} $ lie in the same $\mathcal{R}$-class and are in the same row of the ``egg-box'' for $D$.  So if $d \in {}_iH_j $, then $dm \in {}_iH_k $ for some $k$, proving i.
Write $d = a_i \phi _H^R (g_d )b_j $ and $dm = a_i \phi _H^R (g_{dm} )b_k $ for some $g_d ,g_{dm}  \in G_H^R $.  Then $dm = a_i \phi _H^R (g_d )b_j m = a_i \phi _H^R (g_{dm} )b_k $.  Multiplying by $\bar a_i $ on the left and by $\bar b_k $ on the left gives 
$\phi _H^R (g_d )b_j m\bar b_k  = \phi _H^R (g_{dm} ) \in H$.  So by lemma 2.1, $m^*  \equiv b_j m\bar b_k  \in RT\left( H \right)$ proving ii.  Then $r_{m^ *  } :H \to H$ is a bijection $\phi _H^R \left( g \right)m^ *   = \phi _H^R \left( {gr_{m^ *  } } \right)$ for any $g \in G_H^R $.
Now take any $h = a_i \phi _H^R (g)b_j  \in {}_iH_j $  and compute $hm\bar b_k  = a_i \phi _H^R (g)b_j m\bar b_k  = a_i \phi _H^R (g)m^ *   = a_i \phi _H^R (gr_{m^ *  } )$ Multiplying on the right by $b_k $ then proves iv., and v. then follows by linearity.
For iii., observe that $r_m $ can be written as a composition of bijections: 
$r_m  = {}_i\psi _k  \circ \phi _H^R  \circ \bar r_{m^ *  }  \circ \left( {\phi _H^R } \right)^{ - 1}  \circ \left( {{}_i\psi _j } \right)^{ - 1} $
where $\bar r_{m^ *  } :G_H^R  \to G_H^R $ is the bijection $g \mapsto gr_{m^ *  } $.  
\end{proof}

As an immediate corollary we have

\begin{corollary}  \label{c2.1}
  For any $\mathcal{H}$-class ${}_iH_j  \subseteq D$ and any $m \in M$
either
\begin{itemize}
\item [i.]  $R\left[ {{}_iH_j } \right] \cdot m \subseteq \hat A^D  = \mathop  \oplus \nolimits_{D' < D} R\left[ {D'} \right]$ or
\item  [ii.]  $R\left[ {{}_iH_j } \right] \cdot m = R\left[ {{}_iH_k } \right]$ for some $k$ and if $x = a_i \phi _H^R \left( y \right)b_j $ for some $y \in R\left[ {G_H^R } \right]$, then $xm = a_i \phi _H^R \left( {y\,r_{m^ *  } } \right)b_k $, where $m^ *   = b_j m\bar b_k  \in RT(H)$.
\end{itemize}
\end{corollary}

We obtain the dual versions of proposition \ref{p2.1} and corollary \ref{c2.1} given next by using part ii. of lemmas \ref{l2.1} and \ref{l2.2} along with the formula $m\phi _H^R (g) = \phi _H^R \left( {r_{\bar m} g} \right)$ for $g \in G_H^R $, where $\bar m = \vartheta  \circ \psi \left( {l_m } \right) \in RT(H)$.

\begin{proposition}  \label{p2.2}
  Suppose that for some $d \in D,m \in M$ we have $md \in D$.   For some $i$,$j$, this $d \in {}_iH_j  = {}_i\psi _j \left( {H_\gamma  } \right) = a_i H_\gamma  b_j $ and so $d = a_i \phi _H^R (g_d )b_j $ for some $g_d  \in G_H^R $.  Then:
\begin{itemize}
\item [i.]  $md \in {}_kH_j $ for some $k$.  Define $m^ *   = \bar a_k ma_j $.
\item [ii.]  $m^*  \equiv \bar a_k ma_j  \in LT\left( H \right)$, so $\bar {m^ *  }  = \vartheta  \circ \psi \left( {l_{m^ *  } } \right) \in RT(H)$.
\item [iii.]  The left translation $l_m :h \mapsto mh$ gives a bijection $l_m :{}_iH_j  \to {}_kH_j $.
\item [iv.]  If $h = a_i \phi _H^R (g)b_j  \in {}_iH_j $, then $mh = l_m (h) = a_k \phi _H^R \left( {r_{\bar {m^ *  } } \,g} \right)b_j $.
\item [v.]  For any $x \in R\left[ {{}_iH_j } \right]$, write $x = a_i \phi _H^R \left( y \right)b_j $ for some $y \in R\left[ {G_H^R } \right]$.  Then $mx = a_k \phi _H^R \left( {r_{\bar {m^ *  } } \,\,y} \right)b_j $.
\end{itemize}
\end{proposition}

\begin{corollary}  \label{c2.2}
  For any $\mathcal{H}$-class ${}_iH_j  \subseteq D$ and any $m \in M$
either
\begin{itemize}
\item [i.]  $m \cdot R\left[ {{}_iH_j } \right] \subseteq \hat A^D  = \mathop  \oplus \nolimits_{D' < D} R\left[ {D'} \right]$ or
\item [ii.]  $m \cdot R\left[ {{}_iH_j } \right] = R\left[ {{}_kH_j } \right]$ for some $k$ and if $x = a_i \phi _H^R \left( y \right)b_j $ for some $y \in R\left[ {G_H^R } \right]$, then $mx = a_k \phi _H^R \left( {r_{\bar {m^ *  } } \,\,y} \right)b_j $, where $m^*  \equiv \bar a_k ma_j  \in LT\left( H \right)$ and $\bar {m^ *  }  = \vartheta  \circ \psi \left( {l_{m^ *  } } \right) \in RT(H)$.
\end{itemize}
\end{corollary}

\section{Cell algebra structures on monoid algebras}

     In \cite{May}, a class of algebras called cell algebras is defined which generalize the cellular algebras of Graham and Lehrer \cite{GL}.  These algebras had previously been introduced and studied as ``standardly based algebras'' by Du and Rui in \cite{DR}.  Such algebras share many of the nice properties of cellular algebras.  We will give conditions on a monoid $M$ and domain $R$ such that the monoid algebra $R[M]$ will be a cell algebra and will construct a standard cell basis for such algebras. We first review the definition.
     
	 Let $R$ be a commutative integral domain with unit 1 and let $A$   be an associative, unital $R$-algebra.  Let $\Lambda $ be a finite set with a partial order $ \geqslant $ and for each $\lambda  \in \Lambda $ let $L\left( \lambda  \right),R\left( \lambda  \right)$ be finite sets of ``left indices'' and ``right indices''. Assume that for each $\lambda  \in \Lambda ,s \in L\left( \lambda  \right),{\text{ and }}t \in R\left( \lambda  \right)$ there is an element $\,_s C_t ^\lambda   \in A$ such that the map $(\lambda ,s,t) \mapsto \,_s C_t ^\lambda  $ is injective and $C = \left\{ {_s C_t ^\lambda  :\lambda  \in \Lambda ,s \in L(\lambda ),t \in R\left( \lambda  \right)} \right\}$ is a free $R$-basis for $A$.  Define $R$-submodules of $A$ by $A^\lambda   = R{\text{ - span of }}\left\{ {_s C_t ^\mu  :\mu  \in \Lambda ,\mu  \geqslant \lambda ,s \in L(\mu ),t \in R\left( \mu  \right)} \right\}$ and $\hat A^\lambda   = R{\text{ - span of }}\left\{ {_s C_t ^\mu  :\mu  \in \Lambda ,\mu  > \lambda ,s \in L(\mu ),t \in R\left( \mu  \right)} \right\}$.

\begin {definition} \label {d3.1}  Given $(A,\Lambda,C)$, $A$ is a cell algebra with poset $\Lambda$ and cell basis $C$ if 
    \begin{description}
     \item[i] For any $a \in A,\lambda  \in \Lambda ,{\text{ and }}s,s' \in L\left( \lambda  \right)$, there exists $r_L  = r_L \left( {a,\lambda ,s,s'} \right) \in R$ such that, for any $t \in R\left( \lambda  \right)$, $a \cdot \,_s C_t ^\lambda   = \sum\limits_{s' \in L\left( \lambda  \right)} {r_L  \cdot \,_{s'} C_t ^\lambda  } \, \, \bmod \hat A^\lambda  $, and
      \item[ii] For any $a \in A,\lambda  \in \Lambda ,{\text{ and }}t,t' \in R\left( \lambda  \right)$, there exists  $r_R  = r_R \left( {a,\lambda ,t,t'} \right) \in R$ such that, for any $s \in L\left( \lambda  \right)$, $_s C_t ^\lambda   \cdot a = \sum\limits_{t' \in R\left( \lambda  \right)} {r_R  \cdot \,_s C_{t'} ^\lambda  } \, \,  \bmod \hat A^\lambda  $ .
     \end{description}
\end {definition}

Now for each $\mathcal{D}$-class $D$ in a finite monoid $M$, choose a base element $\gamma _D $ and base $\mathcal{H}$-class $H = H_\gamma  $ as above.  Then define the Schutzenberger group $G_D$ of $D$ to be $G_D = G_H^R $.  (Up to isomorphism, this is independent of the choice of base class $H$.)  Let $R\left[ {G_D} \right]$ be the group algebra of $G_D $ over $R$.
\begin{definition}  \label{d3.2}
A monoid $M$ satifies the $R$-C.A. condition for a given domain $R$ if $R\left[ {G_D } \right]$ has a cell algebra structure for every $\mathcal{D}$-class $D$ in $M$.
\end{definition}

If an $\mathcal{H}$-class $H$ contains an idempotent, then $H$ is actually a subgroup of $M$.  In fact, the maximal subgroups of $M$ are just the $\mathcal{H}$-classes which contain an idempotent.  In this case the group $H$ is isomophic to the Schutzenberger group $G_H^R $.  If $M$ is a \emph{regular} semigroup, then every $\mathcal{D}$-class $D$ contains an idempotent $e$, so $G_D$ is isomorphic to a maximal subgroup $H_e$ of $M$.  Then for regular semigroups $M$ the $R$-C.A. condition is equivalent to requiring that $R[G]$ have a cell algebra structure for every maximal subgroup $G$ of $M$.

If $M$ is a monoid (such as a transformation semigroup $\mathcal{T}_r $ or a partial transformation semigroup $\mathcal{PT}_r $) for which every group $G_D$ is a symmetric group, then the usual Murphy basis gives a cellular (and hence cell) algebra structure to $R[G_D]$ for any domain $R$.  Thus $M$ satisfies the $R$-C.A. condition for any $R$.  

For any finite monoid $M$, if $k$ is a field of characteristic $0$ or characteristic $p$ where $p$ does not divide the order of any $G_D$, then by Maschke's theorem, every $k[G_D]$ is semisimple.  If, in addition, $k$ is algebraically closed, then each $k[G_D]$ is split semisimple.  Such algebras are products of matrix algebras over $k$ and have a natural cellular (hence cell) algebra basis. Thus if $k$ is algebraically closed and of good characteristic relative to $M$, then any $M$ satisfies the $k$-C.A. condition.

Our main result is the following theorem.

\begin{theorem}  \label{t3.1}
  Let $M$ be a finite monoid satifying the $R$-C.A. condition for a domain $R$.  Then $A = R[M]$ is a cell algebra.  The choice of a cell basis for each algebra $R[G_D]$ gives rise to a standard cell basis for $A$.
\end{theorem}

\begin{proof}
  For a given $D \in \mathbb{D}$, put $A_D  = R[G_D ]$ and assume $\Lambda _D ,L_D ,R_D $ define a cell algebra structure on $A_D $ with cell basis \[C_D  = \left\{ {_s C_t ^\lambda  :\lambda  \in \Lambda _D ,s \in L_D (\lambda ),t \in R_D \left( \lambda  \right)} \right\}.\]
  Define a poset $\Lambda $ to consist of all pairs $\left( {D,\lambda } \right)$ where $D$ is a $\mathcal{D}$-class in $M$ and $\lambda  \in \Lambda _D $.  Define the partial order by $(D_1 ,\lambda _1 ) > (D_2 ,\lambda _2 )$ if $D_1  < D_2 $  or $D_1  = D_2 {\text{ and }}\lambda _1  > \lambda _2 $ in $\Lambda _{D_1 } $.    
For $(D,\lambda ) \in \Lambda $, define $L\left( {D,\lambda } \right)$ to be all pairs $\left( {R,s} \right)$ where $R$ is an $\mathcal{R}$-class contained in $D$ and $s \in L_D \left( \lambda  \right)$.  Similarly, define $R\left( {D,\lambda } \right)$ to be all pairs $\left( {L,t} \right)$ where $L$ is an $\mathcal{L}$-class contained in $D$ and $t \in R_D (\lambda )$.  Finally, given $(D,\lambda ) \in \Lambda \,,\,(R,s) \in L(D,\lambda )\,,\,(L,t) \in R(D,\lambda )$, assume $R$ corresponds to row $i$ and $L$ corresponds to column $j$ in the ``egg-box'' for $D$.  Then  define \[_{(R,s)} C_{(L,t)}^{(D,\lambda )}  = {}_i\psi _j \left( {\phi _H^R \left( {_s C_t ^\lambda  } \right)} \right) = a_i \phi _H^R \left( {_s C_t ^\lambda  } \right)b_j .\]  
For fixed $D,R,L$,  since $\left\{ {_s C_t ^\lambda  :\lambda  \in \Lambda _D ,s \in L_D (\lambda ),t \in R_D (\lambda )} \right\}$ is a basis for $R[G_D ]$ by assumption and ${}_i\psi _j {\text{ and }}\phi _H^R $ are bijective,  \[\left\{ {_{(R,s)} C_{(L,t)}^{(D,\lambda )} :\lambda  \in \Lambda _D ,s \in L_D (\lambda ),t \in R_D (\lambda )} \right\}\] will give a basis for $R[{}_iH_j ]$.  Then since $A = R[M]$ is the direct sum of the free submodules $R[H]$ as $H$ varies over all $\mathcal{H}$-classes in $M$, 
\[C = \left\{ {_{(R,s)} C_{(L,t)}^{(D,\lambda )} :(D,\lambda ) \in \Lambda ,(R,s) \in L(D,\lambda ),(L,t) \in R(D,\lambda )} \right\}\]
is a basis for $A$.  To show $C$ is a cell basis, we must check the cell conditions (i) and (ii).
Note first that any element $m$ in a $\mathcal{D}$-class $D_m $ will be in the span of $\left\{ {_{(R,s)} C_{(L,t)}^{(D_m ,\mu )} :\mu  \in \Lambda _{D_m } ,(R,s) \in L(D_m ,\mu ),(L,t) \in R(D_m ,\mu )} \right\}$.  Then if $D_m  < D$ for some $\mathcal{D}$-class $D$, we have $(D_m ,\mu ) > (D,\lambda )$ for any $\mu  \in \Lambda _{D_m} $ and $\lambda \in \Lambda_D $,  so $m \in \hat A^{(D,\lambda )} $.  So $\mathop  \oplus \nolimits_{D' < D} R\left[ {D'} \right] = \hat A^D  \subseteq \hat A^{(D,\lambda )} $.
To  prove (i), we can assume $a$ is a basis element $m \in M$.  Take $(D,\lambda ) \in \Lambda \,,\,(R,s) \in L(D,\lambda )\,,\,(L,t) \in R(D,\lambda )$ and assume $R$ corresponds to row $i$ and $L$ corresponds to column $j$ in the ``egg-box'' for $D$.  Then $_{(R,s)} C_{(L,t)}^{(D,\lambda )}  = {}_i\psi _j \left( {\phi _H^R \left( {_s C_t ^\lambda  } \right)} \right) = a_i \phi _H^R \left( {_s C_t ^\lambda  } \right)b_j  \in R[{}_iH_j ]$.   By corollary \ref{c2.2} there are two cases to consider.

  Case i:  $m \cdot R\left[ {{}_iH_j } \right] \subseteq \mathop  \oplus \nolimits_{D' < D} R\left[ {D'} \right]$ .  Then $m \cdot _{(R,s)} C_{(L,t)}^{(D,\lambda )}  \in m \cdot R[{}_iH_j ] \subseteq \mathop  \oplus \nolimits_{D < D'} R[D'] \subseteq \hat A^{(D,\lambda )} $ and we can satisfy (i) by taking all coefficients $r_L $ to be zero.

Case ii:  $m \cdot R\left[ {{}_iH_j } \right] = R\left[ {{}_kH_j } \right]$ for some $k$.  By corollary 2.2, since $_{(R,s)} C_{(L,t)}^{(D,\lambda )}  = a_i \phi _H^R \left( {_s C_t ^\lambda  } \right)b_j  \in R[{}_iH_j ]$, then $m\,\,_{(R,s)} C_{(L,t)}^{(D,\lambda )}  = a_k \phi _H^R \left( {r_{\bar {m^ *  } } \,_s C_t ^\lambda  } \right)b_j $, where $m^*  \equiv \bar a_k ma_j  \in LT\left( H \right)$
 and $\bar {m^ *  }  = \vartheta  \circ \psi \left( {l_{m^ *  } } \right) \in RT(H)$.  Since $g = r_{\bar {m^ *  } }  \in G_D $, the cell algebra property (i) for $R[G_D ]$ gives $g \cdot \,_s C_t ^\lambda   = \sum\limits_{s' \in L_D \left( \lambda  \right)} {r_L  \cdot \,_{s'} C_t ^\lambda  } \bmod \hat A_D ^\lambda  $, where $r_L $ depends on $s,s',\lambda ,{\text{ and }}m$, but is independent of $t$.  Then 
  \[\begin{aligned}
    m\,\,_{(R,s)} C_{(L,t)}^{(D,\lambda )}  &= a_k \phi _H^R \left( {g\,_s C_t ^\lambda  } \right)b_j  \\
    &= \sum\limits_{s' \in L_D (\lambda )} {r_L  \cdot a_k } \phi _H^R \left( {_{s'} C_t ^\lambda  } \right)b_j \text{ modulo } a_k \phi _H^R \left( {\hat A_D ^\lambda  } \right)b_j 
   \end{aligned}\]
    where \[ a_k \phi _H^R \left( {\hat A_D ^\lambda  } \right)b_j \subseteq {\text{span}}\left\{ {_{(R,s)} C_{(L,t)}^{(D,\mu )} :\mu  > \lambda ,s \in L_D (\mu ),t \in R_D (\mu )} \right\} \subseteq \hat A^{(D,\lambda )} .\]
      But $a_k \phi _H^R \left( {_{s'} C_t ^\lambda  } \right)b_j  = \,_{(R',s')} C_{(L,t)}^{(D,\lambda )} $ where $R'$ is the $\mathcal{R}$-class corresponding to row $k$ of the ``egg-box''.  Then $m \cdot \,_{(R,s)} C_{(L,t)}^{(D,\lambda )}  = \sum\limits_{s' \in L_D (\lambda )} {r_L  \cdot a_k } \phi _H^R \left( {_{s'} C_t ^\lambda  } \right)b_j  = \sum\limits_{s' \in L_D (\lambda )} {r_L  \cdot \,_{(R',s')} C_{(L,t)}^{(D,\lambda )} } \,\,\bmod \hat A^{(D,\lambda )} $.  Since $r_L $ is independent of $L$ and $t$, this yields property (i) for this case ii.
   
  The proof of condition (ii) is parallel.  Thus $C$ is a cell basis and $A = R[M]$ is a cell algebra. 
\end{proof}

If $M$ is a finite monoid satisfying the $R$-C.A. condition, we will assume a fixed cell algebra structure is given to each $R[G_D]$.  We will then call the cell algebra structure obtained in the proof of theorem \ref{t3.1} the \emph{standard cell algebra structure} on $R[M]$.

In \cite{East}, \cite{Wil}, and \cite{GX}, East, Wilcox, and Guo and Xi worked with finite regular semigroups with cellular algebra (hence cell algebra) structure on $R[G]$ for maximal subgroups $G$ of $M$.  Their examples therefore satisfy the $R$-C.A. condition and can be seen to be cell algebras without considering the complicated involution requirements involved in showing a cellular algebra structure.  In \cite{East}, East gives examples of inverse semigroups with cellular $R[G]$ for all maximal $G$ which lack an appropriate involution and therefore do not have cellular $R[M]$.  These examples would be cell algebras by theorem \ref{t3.1}.

More typical examples of cell algebras that are not cellular are the algebras $R[M]$ for $M$ a transformation semigroup $\mathcal{T}_r $ or partial transformation semigroup $\mathcal{PT}_r $.  These were shown to be cell algebras in \cite{May}.  For these examples, a $\mathcal{D}$-class $D$ consists of mappings of a given rank $i$ and the group $G_D$ is the symmetric group $\mathfrak{G}_i$.  Since, as remarked above, the symmetric groups have cellular (hence cell) structures on their group algebras, theorem \ref{t3.1} applies and provides a cell algebra structure on $M$.

\section{Properties of the cell algebra $A = R[M]$}

In this section we assume that $M$ is a finite monoid satisfying the $R$-C.A. condition, that is, such that for every $D \in \mathbb{D}$ the group algebra $R[G_D ]$ of the Schutzenberger group for $D$ is a cell algebra.  Then by Theorem \ref{t3.1}, $A = R[M]$ is a cell algebra and we can apply the results in \cite{May} to $A$ with the standard cell algebra structure.

For $(D,\lambda ) \in \Lambda $, the left cell module $\,_L C^{(D,\lambda )} $ is a left $A$-module which is a free $R$-module with basis $\left\{ {_{(R,s)} C^{(D,\lambda )} :(R,s) \in L\left( {D,\lambda } \right)} \right\}$.   Similarly, the right cell module $C_R ^{(D,\lambda )} $ for $(D,\lambda )$ is a right $A$-module and a free $R$-module with basis $\left\{ {C_{(L,t)} ^{(D,\lambda )} :(L,t) \in R\left( {D,\lambda } \right)} \right\}$.  For each $(D,\lambda ) \in \Lambda $ there is an $R$-bilinear map $ \left\langle - , - \right\rangle \,:\,\left( {C_R ^{(D,\lambda )} ,\,_L C^{(D,\lambda )} } \right) \to R$ defined by the property $\left( {_{(R',s')} C_{(L,t)}^{(D,\lambda )} } \right) \cdot \left( {_{(R,s)} C_{(L',t')}^{(D,\lambda )} } \right) = \left\langle {C_{(L,t)}^{(D,\lambda )} ,\,_{(R,s)} C^{(D,\lambda )} } \right\rangle \,_{(R',s')} C_{(L',t')}^{(D,\lambda )} \,\bmod \,\hat A^{(D,\lambda )} $
  for any choice of $R',s',L',t'$.
  
Right and left radicals are defined by 
\[{\text{rad}}\left( {C_R^{(D,\lambda )} } \right) = \left\{ {x \in C_R^{(D,\lambda )} :\left\langle {x,y} \right\rangle  = 0{\text{ for all }}y \in \,_L C^{(D,\lambda )} } \right\}\]
\[{\text{rad}}\left( {\,_L C^{(D,\lambda )} } \right) = \left\{ {y \in \,_L C^{(D,\lambda )} :\left\langle {x,y} \right\rangle  = 0{\text{ for all }}x \in C_R^{(D,\lambda )} } \right\}.\]
Then define $D_R^{(D,\lambda )}  = \frac{{C_R^{(D,\lambda )} }}
{{{\text{rad}}\left( {C_R^{(D,\lambda )} } \right)}}$ and  $\,_L D^{(D,\lambda )}  = \frac{{\,_L C^{(D,\lambda )} }}
{{{\text{rad}}\left( {\,_L C^{(D,\lambda )} } \right)}}$.   Finally, define $\Lambda _0  = \left\{ {(D,\lambda ) \in \Lambda :\left\langle {x,y} \right\rangle  \ne 0{\text{ for some }}x \in C_R^{(D,\lambda )} ,y \in \,_L C^{(D,\lambda )} } \right\}$.  Evidently, $\lambda  \in \Lambda _0  \Leftrightarrow D_R^{(D,\lambda )}  \ne 0 \Leftrightarrow \,_L D^{(D,\lambda )}  \ne 0$.  A major result of \cite{May} is

\begin{theorem}   \label{t4.1}
  Assume $R = k$ is a field.  Then \\
  (a) $\left\{ {D_R ^{(D,\mu )} :(D,\mu ) \in \Lambda _0 } \right\}$ is a complete set of pairwise inequivalent irreducible right $A$-modules and \\ (b) $\left\{ {\,_L D^{(D,\mu )} :(D,\mu ) \in \Lambda _0 } \right\}$ is a complete set of pairwise inequivalent irreducible left $A$-modules.
\end{theorem}

To find $\Lambda _0 $ for $A$ we need to relate the bracket $\left\langle { - , - } \right\rangle $ on $A$ to the brackets $\left\langle { - , - } \right\rangle _D $ on the cell algebras $A_D  = R[G_D ]$.  For a given $\mathcal{D}$-class $D$, write $R_i $ for the $\mathcal{R}$-class corresponding to row $i$ of the ``egg-box'' and $L_j $ for the $\mathcal{L}$-class corresponding to column $j$.  

\begin{definition}  \label{d4.1}
$R_i \,,\,L_j $ are matched if there exist $x \in L_j \,,\,y \in R_i $ such that $xy \in D$.  $R_i \,,\,L_j $ are unmatched if $L_j R_i  \cap D = \emptyset $.
\end{definition}

\begin{lemma}  \label{l4.1}
Assume $R_i \,,\,L_j $ are matched.  Write $x = a_{i'} \phi _H^R (\xi )b_j  \in R[L_j ]$ and $y = a_i \phi _H^R (\eta )b_{j'}  \in R[R_i ]$ for any $\xi ,\eta  \in R[G_D ]$.  Then $xy = a_{i'} \phi _H^R \left( {\xi r_{m(i,j)} \eta } \right)b_{j'} $ where $m(i,j) \equiv b_j a_i \gamma  \in RT(H)$.
\end{lemma}

\begin{proof}
If $R_i \,,\,L_j $ are matched, then for some choice of $i',j'$ and some elements $x \in {}_{i'}H_j  \subseteq L_j \,,\,y \in {}_iH_{j'}  \subseteq R_i $ we have $xy \in {}_{i'}H_{j'}  \subseteq D$.  But then (for the given $i',j'$) we have $xy \in {}_{i'}H_{j'}  \subseteq D$ for \emph{any} $x \in {}_{i'}H_j \,,\,y \in {}_iH_{j'} $ by propositions \ref{p2.1} and \ref{p2.2}.  In particular, for $x = a_{i'} \gamma b_j  \in {}_{i'}H_j \,,\,y = a_i \gamma b_{j'}  \in {}_iH_{j'} $ we get $xy = a_{i'} \gamma b_j a_i \gamma b_{j'}  \in {}_{i'}H_{j'} $.  Then multiplying by $\bar a_{i'} $ on the left and $\bar b_{j'} $ on the right gives $\gamma b_j a_i \gamma  \in H$.  But then $m(i,j) \equiv b_j a_i \gamma  \in RT(H)$ by lemma \ref{l2.1}.  

Now consider any $i',j'$ and any $g = r_m ,h = r_n  \in G_D $ and let
\[x = a_{i'} \phi _H^R (g)b_j  \in R[L_j ]\,,\,y = a_i \phi _H^R (h)b_{j'}  \in R[R_i ].\]
Then 
\[
   \begin{aligned}
     xy &= (a_{i'} \phi _H^R \left( g \right)b_j )(a_i \phi _H^R \left( h \right)b_{j'} ) = a_{i'} \gamma mb_j a_i \gamma nb_{j'} \\
      &= a_{i'} \gamma m \cdot m(i,j) \cdot nb_{j'}  = a_{i'} \phi _H^R \left( {r_m r_{m(i,j)} r_n } \right)b_{j'} \\
       &= a_{i'} \phi _H^R \left( {gr_{m(i,j)} h} \right)b_{j'} .
   \end{aligned}
\]
Then by linearity of $\phi _H^R $, $xy = a_{i'} \phi _H^R \left( {\xi r_{m(i,j)} \eta } \right)b_{j'} $ for arbitrary $\xi ,\eta  \in R[G_D ]$ when $x = a_{i'} \phi _H^R (\xi )b_j  \in R[L_j ]\,,\,y = a_i \phi _H^R (\eta )b_{j'}  \in R[R_i ]$,  proving the lemma.
\end{proof}

For $i \in \left\{ {1,2, \cdots ,n(D,R)} \right\}$, define $\left( {_L C^{(D,\lambda )} } \right)_i $
 to be the free $R$-module with basis $\left\{ {_{(R_i ,s)} C^{(D,\lambda )} :\lambda  \in \Lambda _D ,s \in L_D (\lambda )} \right\}$, so $\,_L C^{(D,\lambda )}  = \mathop  \oplus \limits_i \left( {_L C^{(D,\lambda )} } \right)_i $.  Similarly, for $j \in \left\{ {1,2, \cdots ,n(D,L)} \right\}$, define $\left( {C_R ^{(D,\lambda )} } \right)_j $ to be the free $R$-module with basis $\left\{ {C_{(L_j ,t)}^{(D,\lambda )} :\lambda  \in \Lambda _D ,t \in R_D (\lambda )} \right\}$, so $C_R^{(D,\lambda )}  = \mathop  \oplus \limits_j \left( {C_R^{(D,\lambda )} } \right)_j $.  Notice that for each $i$, $\phi _i :\,_{(R_i ,s)} C^{(D,\lambda )}  \mapsto \,_s C^\lambda  $ gives an isomorphism (of $R$-modules) $\phi _i :\left( {_L C^{(D,\lambda )} } \right)_i  \to \,_L C^\lambda  $.  Similarly, $\phi _j :C_{(L_j ,t)}^{(D,\lambda )}  \mapsto C_t^\lambda  $ gives an isomorphism $\phi _j :\left( {C_R ^{(D,\lambda )} } \right)_j  \to C_R ^\lambda  $.

\begin{proposition}   \label{p4.1}
  Take $X \in \left( {C_R ^{(D,\lambda )} } \right)_j \,,\,Y \in \left( {_L C^{(D,\lambda )} } \right)_i $.  Then
\begin{enumerate} [\upshape(a)]
\item If $R_i \,,\,L_j $ are not matched, then $\left\langle {X,Y} \right\rangle  = 0$,
\item If $R_i \,,\,L_j $ are matched, then 
\[\left\langle {X,Y} \right\rangle  = \left\langle {\phi _j (X),r_{m(i,j)} \phi _i \left( Y \right)} \right\rangle _D  = \left\langle {r_{m(i,j)} \phi _j (X),\phi _i \left( Y \right)} \right\rangle _D \]
 where $m(i,j) \equiv b_j a_i \gamma  \in RT(H)$.
\end{enumerate}
\end{proposition}

\begin{proof}
  It suffices to check (a) and (b) for basis elements $X,Y$, so assume $X = C_{(L_j ,t)}^{(D,\lambda )} \,,\,Y = \,_{(R_i ,s)} C^{(D,\lambda )} $ and put $x = _{(R_{i'} ,s')} C_{_{(L_j ,t)} }^{(D,\lambda )}  = a_{i'} \phi _H^R \left( {_{s'} C_t ^\lambda  } \right)b_j  \in L_j $ and $y = _{(R_i ,s)} C_{_{(L_{j'} ,t')} }^{(D,\lambda )}  = a_i \phi _H^R \left( {_s C_{t'} ^\lambda  } \right)b_{j'}  \in R_i $.  Then $xy = \left\langle {X,Y} \right\rangle \,_{(R_{i'} ,s')} C_{(L_{j'} ,t')}^{(D,\lambda )} \bmod \hat A^{(D,\lambda )} $.

If  $R_i \,,\,L_j $ are not matched, then $xy \in \hat A^D  \subseteq \hat A^{(D,\lambda )} $, so $\left\langle {X,Y} \right\rangle  = 0$, proving (a).

If  $R_i \,,\,L_j $ are matched, then, by lemma \ref{l4.1}, 
\[ \begin{aligned}
xy &= a_{i'} \phi _H^R \left( {_{s'} C_t ^\lambda   \cdot r_{m(i,j)}  \cdot \,_s C_{t'} ^\lambda  } \right)b_{j'}  \\
   &= a_{i'} \phi _H^R \left( {\left\langle {C_t^\lambda  ,r_{m(i,j)}  \cdot \,_s C^\lambda  } \right\rangle _D \,_{s'} C_{t'} ^\lambda  } \right)b_{j'} \bmod \hat A^{(D,\lambda )}  \hfill \\
    &= \left\langle {C_t^\lambda  ,r_{m(i,j)}  \cdot \,_s C^\lambda  } \right\rangle _D  \cdot a_{i'} \phi _H^R \left( {_{s'} C_{t'} ^\lambda  } \right)b_{j'} \bmod \hat A^{(D,\lambda )}  \hfill \\
    &= \left\langle {\phi _j (X),r_{m(i,j)}  \cdot \phi _i (Y)} \right\rangle _D  \cdot _{(R_{i'} ,s')} C_{(L_{j'} ,t')}^{(D,\lambda )} \bmod \hat A^{(D,\lambda )}  \hfill .
 \end{aligned}
 \]
This gives $\left\langle {X,Y} \right\rangle  = \left\langle {\phi _j (X),r_{m(i,j)} \phi _i \left( Y \right)} \right\rangle _D  = \left\langle {r_{m(i,j)} \phi _j (X),\phi _i \left( Y \right)} \right\rangle _D $, proving part (b).
\end{proof}

We note the following corollary for future use.

\begin{corollary}  \label{c4.0}

  Let $M$ be a finite monoid satisfying the $R$-C.A. condition and place the standard cell algebra structure on $R[M]$.  Then for any $\lambda  \in \Lambda _D \,,\,D \in \mathbb{D}$:
\begin{enumerate}
\item  If $\rad_D \left( {\,_L C^\lambda  } \right) \ne 0$, then $\rad\left( {\,_L C^{(D,\lambda )} } \right) \ne 0$,
\item  If $\rad_D \left( {D_R ^\lambda  } \right) \ne 0$, then $\rad\left( {D_R ^{(D,\lambda )} } \right) \ne 0$.
\end{enumerate}
\end{corollary}

\begin{proof}
  Assume $\rad_D \left( {\,_L C^\lambda  } \right) \ne 0$ and take a $y \ne 0$ in $\rad_D \left( {\,_L C^\lambda  } \right)$.  Write $y = \sum\nolimits_{s \in L(\lambda )} {c(s) \cdot \,_s C^\lambda  } $ and put $Y = \sum\nolimits_{s \in L(\lambda )} {c(s)\,_{(R_i ,s)} C^{(D,\lambda )} }  \in \left( {\,_L C^{(D,\lambda )} } \right)_i  \subseteq \,_L C^{(D,\lambda )} $ for some $R_i  \subseteq D$.  Then $Y \ne 0$ and we claim $Y \in \rad\left( {\,_L C^{(D,\lambda )} } \right)$:  Take any $X \in \left( {C_R ^{(D,\lambda )} } \right)_j $ .  Then by the proposition, if  $R_i \,,\,L_j $ are not matched we have $\left\langle {X,Y} \right\rangle  = 0$, while if $R_i \,,\,L_j $ are matched, then \[\left\langle {X,Y} \right\rangle  = \left\langle {r_{m(i,j)} \phi _j (X),\phi _i \left( Y \right)} \right\rangle _D  = \left\langle {r_{m(i,j)} \phi _j (X),y} \right\rangle _D  = 0\]
   since $y \in \rad_D \left( {\,_L C^\lambda  } \right)$.  Then $\left\langle {X,Y} \right\rangle  = 0$ for any $X \in C_R ^{(D,\lambda )} $ and $Y \in \rad\left( {\,_L C^{(D,\lambda )} } \right)$ as claimed.  The proof of ii. is parallel.  
\end{proof}

Write $D^2  = \left\{ {xy:x,y \in D} \right\}$ and recall that $\hat A^D  = \mathop  \oplus \nolimits_{D' < D} R[D'] \subseteq \hat A^{(D,\lambda )} $ for any $\lambda  \in \Lambda _D $.

\begin{corollary}   \label{c4.2}
For any $\mathcal{D}$-class $D$,
\begin{enumerate}  [\upshape(a)]
\item   If $D^2  \subseteq \hat A^D $, then $(D,\lambda ) \notin \Lambda _0 $
 for any $\lambda  \in \Lambda _D $
\item   If $D^2  \not\subset \hat A^D $, then $(D,\lambda ) \in \Lambda _0  \Leftrightarrow \lambda  \in \left( {\Lambda _D } \right)_0 $.
\end{enumerate}
\end{corollary}

\begin{proof}
 $D^2  \subseteq \hat A^D $ if and only if no pair $R_i \,,\,L_j $ are matched. 

(a) $D^2  \subseteq \hat A^D $ implies no pair $R_i \,,\,L_j $ is matched, so by proposition \ref{p4.1} $\left\langle {X,Y} \right\rangle  = 0$ for every $X \in C_R^{(D,\lambda )} ,Y \in \,_L C_{}^{(D,\lambda )} $.  Then ${\text{rad}}\left( {C_R^{(D,\lambda )} } \right) = C_R^{(D,\lambda )} $ and  $(D,\lambda ) \notin \Lambda _0 $.

 (b)  If $D^2  \not\subset \hat A^D $ then $R_i \,,\,L_j $ are matched for at least one pair $i,j$.

For $ \Rightarrow $:  If $\left( {D,\lambda } \right) \in \Lambda _0 $, then  $\left\langle {X,Y} \right\rangle  \ne 0$ for some $X \in \left( {C_R ^{(D,\lambda )} } \right)_j \,,\,Y \in \left( {_L C^{(D,\lambda )} } \right)_i $, where $i,j$ must be matched.  Then by proposition \ref{p4.1}, $\left\langle {X,Y} \right\rangle  = \left\langle {\phi _j (X),r_{m(i,j)} \phi _i \left( Y \right)} \right\rangle _D  \ne 0$, where $\phi _j (X) \in C_R ^\lambda  \,,\,r_{m(i,j)} \phi _i (Y) \in \,_L C^\lambda  $.  So ${\text{rad}}_D \left( {C_R^\lambda  } \right) \ne C_R^\lambda  $ and $\lambda  \in \left( {\Lambda _D } \right)_0 $.

For $ \Leftarrow $:  If $\lambda  \in \left( {\Lambda _D } \right)_0 $, then there exist $\xi  \in C_R ^\lambda  \,,\,\eta  \in \,_L C^\lambda  $ with $\left\langle {\xi ,\eta } \right\rangle _D  \ne 0$.  Choose $i,j$ with $R_i \,,\,L_j $ matched and let $X = \phi _j ^{ - 1} (\xi ) \in \left( {C_R ^{(D,\lambda )} } \right)_j \,,\,Y = \phi _i ^{ - 1} (r_{m(i,j)} ^{ - 1}  \cdot \eta ) \in \left( {_L C^{(D,\lambda )} } \right)_i $.  (Note that $r_{m(i,j)}  \in G_D $, so $r_{m\left( {i,j} \right)} ^{ - 1}  \in G_D $ is well defined.)   Then by proposition \ref{p4.1}, $\left\langle {X,Y} \right\rangle  = \left\langle {\phi _j (X),r_{m(i,j)} \phi _i \left( Y \right)} \right\rangle _D  = \left\langle {\xi ,\eta } \right\rangle _D  \ne 0$.  Then ${\text{rad}}\left( {C_R^{(D,\lambda )} } \right) \ne C_R^{(D,\lambda )} $ and $(D,\lambda ) \in \Lambda _0 $.
\end{proof}

Consider the question of whether a monoid algebra $k[M]$ is quasi-hereditary (as defined by \cite{CPS}).  The following was proved in \cite{May}:

\begin{proposition}    \label{p4.2}
 Let $k$ be a field and $A$ be a cell algebra over $k$ such that $\Lambda _0  = \Lambda $.  Then $A$ is quasi-hereditary.
\end{proposition}

Then by corollary \ref{c4.2} we have 

\begin{corollary}   \label{c4.3}
  Let $k$ be a field and $M$ a finite monoid such that for every class $D \in \mathbb{D}$ , $D^2  \not\subset \hat A^D $ and the group algebra $k[G_D ]$ of the Schutzenberger group for $D$ is a cell algebra with $\left( {\Lambda _D } \right)_0  = \Lambda _D $.  Then $A = k[M]$ is quasi-hereditary.
\end{corollary}

Recall that a semigroup $M$ is \emph{regular} if for every $x \in M$ there is a $y \in M$ such that $x = xyx$.  In a regular semigroup, every $\mathcal{D}$-class $D$ contains an idempotent $\rho $.  Then $\rho  = \rho ^2  \in D \cap D^2 $ , so $D^2  \not\subset \hat A^D $.  Thus

\begin{corollary}   \label{c4.4}
  Let $k$ be a field and $M$ a regular finite monoid such that for every class $D \in \mathbb{D}$ , the group algebra $k[G_D ]$ of the Schutzenberger group for $D$ is a cell algebra with $\left( {\Lambda _D } \right)_0  = \Lambda _D $.  Then $A = k[M]$ is quasi-hereditary.
\end{corollary}

Note that in a regular semi-group the base $\mathcal{H}$-class $H$ for any class $D$ can be chosen to contain an idempotent.  Then $H$ is a (maximal) subgroup of $M$ with $H$ isomorphic to $G_D $.  So the condition in corollary \ref{c4.4} can be replaced by the requirement that for any maximal subgroup $G$ of $M$ the group algebra $k[G]$ must be a cell algebra with $\left( {\Lambda _G } \right)_0  = \Lambda _G $. 

As a special case we obtain the following result of Putcha \cite{Pu}.

\begin{corollary}   \label{c4.45}
  For any regular finite monoid $M$, the complex monoid algebra $A=\mathbb{C}[M]$ is quasi-hereditary.
\end{corollary}

\begin{proof}  As remarked above, each $\mathbb{C}[G_D]$ will be (split) semisimple by Maschke's theorem.  As a product of complex matrix algebras it has a natural cellular basis and (being semisimple) is quasi-hereditary.  Each $\mathbb{C}[G_D]$ then satisfies the conditions of \ref{c4.4}, so $A$ is quasi-hereditary. 
\end{proof}

Now consider the question of semisimplicity for a monoid algebra $k[M]$.  The following criterion for semisimplicity of a cell algebra follows from results in \cite{May}.

\begin{proposition}    \label{p4.3}
  A cell algebra $A$ over a field $k$ is semisimple if and only if for every $\lambda  \in \Lambda $ we have $\,_L C^\lambda   = \,_L D^\lambda  $ and $C_R ^\lambda   = D_R ^\lambda  $.
\end{proposition}

\begin{proof}
  Assume $A$ is semisimple.  Then for every $\lambda  \in \Lambda _0 $ we have $\,_L P^\lambda   = \,_L D^\lambda  $, where $\,_L P^\lambda  $ is the principle indecomposable left module corresponding to the irreducible $\,_L D^\lambda  $.  But since  $\,_L P^\lambda  $ always has a filtration with $\,_L C^\lambda  $ as the ``top'' quotient, we must have $\,_L P^\lambda   = \,_L C^\lambda   = \,_L D^\lambda  $.  Since each $\,_L D^\lambda  $ is absolutely irreducible, the multiplicity of  $\,_L P^\lambda  $ as a direct summand in $A$ is just $\dim \left( {\,_L D^\lambda  } \right) = \dim \left( {\,_L C^\lambda  } \right) = \left| {L(\lambda )} \right|$.  So $\dim \left( A \right) = \sum\nolimits_{\lambda  \in \Lambda _0 } {\left| {L(\lambda )} \right|^2 } $.  Similarly, for every $\lambda  \in \Lambda _0 $ we have $P_R ^\lambda   = C_R ^\lambda   = D_R ^\lambda  $ (where $P_R ^\lambda  $ is the principle indecomposable right module corresponding to the irreducible $D_R ^\lambda  $) and $\dim \left( A \right) = \sum\nolimits_{\lambda  \in \Lambda _0 } {\left| {R(\lambda )} \right|^2 } $.  But the multiplicity of $\,_L P^\lambda  $ as a direct summand in $A$ and the multiplicity of $P_R ^\lambda  $ as a direct summand in $A$ must be equal (as both equal the number of primitive idempotents corresponding to $\lambda  \in \Lambda _0 $).  So $\left| {L(\lambda )} \right| = \left| {R(\lambda )} \right|$  and we can write $\dim \left( A \right) = \sum\nolimits_{\lambda  \in \Lambda _0 } {\left| {L(\lambda )} \right| \cdot \left| {R(\lambda )} \right|} $.  On the other hand, a direct count of basis elements $_s C_t ^\lambda  ,\,\lambda  \in \Lambda \,,\,s \in L(\lambda )\,,\,t \in R(\lambda )\,,$ shows that $\dim \left( A \right) = \sum\nolimits_{\lambda  \in \Lambda } {\left| {L(\lambda )} \right| \cdot \left| {R(\lambda )} \right|} . $  It follows that $\Lambda _0  = \Lambda $, so  $\,_L C^\lambda   = \,_L D^\lambda  $ and $C_R ^\lambda   = D_R ^\lambda  $ for every $\lambda  \in \Lambda $.

 Now assume $\,_L C^\lambda   = \,_L D^\lambda  $ and $C_R ^\lambda   = D_R ^\lambda  $ for every $\lambda  \in \Lambda $, so $\Lambda _0  = \Lambda $.  As in \cite{May}, for $\mu  \in \Lambda _0 \,,\,\lambda  \in \Lambda $, let $Rd_{\lambda \mu }  = [C_R ^\lambda  :D_R ^\mu  ]$ be the multiplicity of the irreducible $D_R ^\mu  $ as a composition factor in $C_R ^\lambda  $, let $Ld_{\lambda \mu }  = \left[ {\,_L C^\lambda  :\,_L D^\mu  } \right]$ be the multiplicity of $\,_L D^\mu  $ in $\,_L C^\lambda  $, and write $RD = \left( {Rd_{\lambda \mu } } \right)\,,\,\,\,LD = \left( {Ld_{\lambda \mu } } \right)$ for the decomposition matrices of $A$.  Then $RD$ and $LD$ are both square identity matrices.   The right Cantor matrix, $RC$, and left Cantor matrix, $LC$, are the square $\left| {\Lambda _0 } \right| \times \left| {\Lambda _0 } \right|$ matrices where for $\lambda \,,\,\mu  \in \Lambda _0 $, $RC_{\lambda \mu }  = \left[ {P_R ^\lambda  :D_R ^\mu  } \right]$ and $LC_{\lambda \mu }  = \left[ {\,_L P^\lambda  :\,_L D^\mu  } \right]$.   As shown in \cite{May}, we have $RC = LD^T  \cdot RD$ and $LC = RD^T  \cdot LD$ (where $T$ denotes the transpose matrix).  Then $RC$ and $LC$ are also identity matrices and we must have $P_R ^\lambda   = D_R ^\lambda  $ and $\,_L P^\lambda   = \,_L D^\lambda  $ for every $\lambda  \in \Lambda  = \Lambda _0 $.  Since $A$ is a direct sum of principle indecomposable left modules isomorphic to the various$\,_L P^\lambda   = \,_L D^\lambda  $, it is a direct sum of irreducible modules and therefore semisimple.  
\end{proof}

Consider a finite monoid $M$ satisfying the $R$-C.A. condition and place the standard cell algebra structure on $M$.

\begin{corollary}  \label{c4.41}
If $k[M]$ is semisimple, then $k[G_D]$ is semisimple for every $D \in \mathbb{D} $.
\end{corollary}

\begin{proof}
If $k[G_D]$ is not semisimple for some $D$, then by the proposition either $\rad_D \left( {\,_L C^\lambda  } \right) \ne 0$ or $\rad_D \left( {D_R ^\lambda  } \right) \ne 0$.  Assume $\rad_D \left( {\,_L C^\lambda  } \right) \ne 0$ (if $\rad_D \left( {D_R ^\lambda  } \right) \ne 0$ the proof is similar). Then by corollary \ref{c4.0} we have $\rad\left( {\,_L C^{(D,\lambda )} } \right) \ne 0$.  But then $k[M]$ is not semisimple by proposition \ref{p4.3}. 
\end{proof}

If all the algebras $k[G_D ]$ are semisimple, we would like a condition guaranteeing that $k[M]$ itself is semisimple.
For any $\mathcal{D}$-class $D$, let $\mathbb{L}(D)$ be the set of all $\mathcal{L}$-classes contained in $D$ and $\mathbb{R}(D)$ be the set of all $\mathcal{R}$-classes contained in $D$.  Define the ``bijection condition'' on $D$:

\textbf{Bijection condition:}  There exists a bijection $F:\mathbb{L}(D) \to \mathbb{R}(D)$ such that $L$ and $F(L)$ are matched while $L$ and $R$ are not matched if $R \ne F(L)$.  

\begin{proposition}    \label{p4.4}
  Let $k$ be a field and $M$ a monoid such that for every $\mathcal{D}$-class $D$
\begin{enumerate}
\item   $k[G_D ]$ is a semisimple cell algebra and
\item   the bijection condition is satisfied for $D$.
\end{enumerate}
Then $k[M]$ is a semisimple cell algebra.
\end{proposition}

\begin{proof}
  Place the standard cell algebra structure on $k[M]$ as in theorem \ref{t3.1}.  We will use proposition \ref{p4.3} to show semisimplicity.  Take $\left( {D,\lambda } \right) \in \Lambda $.  We will show ${\text{rad}}\left( {\,_L C^{(D,\lambda )} } \right) = 0$ and therefore $\,_L C^{(D,\lambda )}  = \,_L D^{(D,\lambda )} $.  Take any $Y = \sum\nolimits_k {Y_k }  \in {\text{rad}}\left( {\,_L C^{(D,\lambda )} } \right)$ where $Y_k  \in \left( {_L C^{(D,\lambda )} } \right)_k $ .  Then $\left\langle {X,Y} \right\rangle  = 0$ for any $X \in C_R ^{(D,\lambda )} $.  If $X \in \left( {C_R ^{(D,\lambda )} } \right)_j $, then $\left\langle {X,Y_k } \right\rangle  = 0$ whenever $L_j  \ne F(R_k )$ by proposition \ref{p4.1}.  Then we must also have $\left\langle {X,Y_k } \right\rangle  = 0$ when $L_j  = F(R_k )$.  But when $L_j ,R_k $ are matched, proposition \ref{p4.1} gives $\left\langle {X,Y_k } \right\rangle  = \left\langle {\phi _j (X),r_{m(i,j)} \phi _i \left( {Y_k } \right)} \right\rangle _D  = 0$.  Since $\phi _j (X)$ is an arbitrary basis element in $C_R ^\lambda  $, we have $r_{m(i,j)} \phi _i \left( {Y_k } \right) \in {\text{rad}}_D \left( {\,_L C^\lambda  } \right)$.  But since $k[G_D ]$ is semisimple, ${\text{rad}}_D \left( {_L C^\lambda  } \right) = 0$ by proposition \ref{p4.3}.  Then $r_{m(i,j)} \phi _i \left( {Y_k } \right) = 0$, and since $r_{m\left( {i,j} \right)}  \in G_D $ is invertible, $Y_k  = 0$.  Since this is true for each $k$, we have $Y = 0$.  So ${\text{rad}}\left( {_L C^{(D,\lambda )} } \right) = 0$.  A parallel argument shows that ${\text{rad}}\left( {C_R ^{(D,\lambda )} } \right) = 0$, so $k[M]$ is semisimple by proposition \ref{p4.3}.
\end{proof}

Recall that an \emph{inverse} of an element $a$ in a semi-group is an element $a^{ - 1} $ such that $aa^{ - 1} a = a$ and $a^{ - 1} aa^{ - 1}  = a^{ - 1} $.  An \emph{inverse semi-group} is a semi-group in which each element has a unique inverse.  A standard result in semi-group theory is that any inverse semi-group satisfies the bijection condition.  (In fact each $\mathcal{L}$-class and each $\mathcal{R}$-class contains a unique idempotent.  Given an $\mathcal{L}$-class $L$, the unique matching $\mathcal{R}$-class $R = F(L)$ is the class containing the same idempotent as $L$.)  Then proposition \ref{p4.4} and corollary \ref{c4.41} yield

\begin{corollary}    \label{c4.5}
  Let $k$ be a field and $M$ a finite monoid which is an inverse semi-group and satisfies the $k$-C.A. condition.  Then $k[M]$ is semisimple if and only if $k[G_D]$ is semisimple for every $\mathcal{D}$-class $D$.   
\end{corollary}

For an inverse semi-group the base $\mathcal{H}$-class $H$ for any class $D$ can be chosen to be a (maximal) subgroup of $M$ with $H$ isomorphic to $G_D $.  So the condition in corollary \ref{c4.5} can be replaced by the requirement that for any maximal subgroup $G$ of $M$ the group algebra $k[G]$ must be a semisimple cell algebra.

It is well known that if the finite monoid $M$ is an inverse semi-group and the field $k$ has characteristic not dividing the order of any $G_D$, then $k[M]$ is semisimple (see e.g. \cite{CP}, which cites \cite{Og}).  In this case each $k[G_D]$ is semisimple by Maschke's theorem.  If $k$ is also algebraically closed, then as remarked above, $M$ satisfies the $k$-C.A. condition.  Thus corollary \ref{c4.5} yields the semisimplicity of $k[M]$ under the additional hypothesis of algebraic closure for $k$.

\section{Twisted monoid algebras}

A \emph{twisting} on a monoid $M$ (with values in a commutative domain $R$ with unit $1$) is a map $\pi :M \times M \to R$ such that (i) for all $x,y,z \in M$, \[\pi \left( {x,y} \right)\pi \left( {xy,z} \right) = \pi \left( {x,yz} \right)\pi \left( {y,z} \right)\] and   (ii) for all $x \in M$, \[\pi \left( {x,id} \right) = 1 = \pi \left( {id,x} \right)\] (where $id$ is the identity in $M$).

Given a twisting $\pi $ on $M$, define an algebra $R^\pi  \left[ M \right]$ to be the free $R$-module with basis $M$ and multiplication $x \circ y = \pi (x,y)\,xy$ for $x,y \in M$.  Then $R^\pi  \left[ M \right]$ is an associative $R$-algebra with unit $1$.  
     
We would like conditions under which $R^\pi  \left[ M \right]$ will be a cell algebra.  In \cite{Wil} and \cite{GX}, Wilcox and Guo and Xi have investigated when $R^\pi  \left[ M \right]$ can be a cellular algebra.  Much of the difficulty in their analyses involves defining the involution anti-isomorphism $ * $ required for a cellular algebra.  Since cell algebras don't require such a map, the corresponding results are both simpler to obtain and of more general applicability.  We require one ``compatibility'' condition for our twisting:

\begin{definition}    \label{d5.1}
 A twisting $\pi $ on a monoid $M$ is \emph{compatible} if for all $a,x \in M$
\begin{enumerate} 
\item  If  $ax \mathcal{D} x$ , then $\pi \left( {a,y} \right) = \pi \left( {a,x} \right)$ whenever $y \in H_x $,
\item If  $xa\mathcal{D}x$ , then $\pi \left( {y,a} \right) = \pi \left( {x,a} \right)$ whenever $y \in H_x $.
\end{enumerate}
\end{definition}

In \cite{GL}, $\pi $ is defined to be an $\mathcal{LR}$-twisting if $x\mathcal{L}y \Rightarrow \pi \left( {x,z} \right) = \pi \left( {y,z} \right)$ and $y\mathcal{R}z \Rightarrow \pi \left( {x,y} \right) = \pi \left( {x,z} \right)$ for all $x,y,z \in M$.  Clearly any $\mathcal{LR}$-twisting is compatible.

\begin{theorem}    \label{t5.1}

Let $M$ be a finite monoid satisfying the $R$-C.A. condition.  Let $\pi $ be a compatible twisting on $M$.   Then $R^\pi  \left[ M \right]$ has a cell algebra structure with the same $\Lambda ,R,L$ and cell basis $C$ as for the standard cell algebra structure on $R[M]$.
\end{theorem}

\begin{proof}
  Since as $R$-modules $R^\pi  \left[ M \right]$ and $R[M]$ are identical, $C$ is an $R$-basis for $R^\pi  \left[ M \right]$ and we need only check that conditions (i) and (ii) for a cell algebra are satisfied for the new multiplication.

  For (i), assume $a$ is a basis element, $a \in M$, and take any $_{(R,s)} C_{(L,t)}^{(D,\lambda )}  \in C$.  Then $_{(R,s)} C_{(L,t)}^{(D,\lambda )}  \in R[H]$ is an $R$-linear combination of elements in an $\mathcal{H}$-class $H = L \cap R \subseteq D$.  By corollary \ref{c2.2} we have two possibilities: I.  $a \cdot R\left[ H \right] \subseteq \hat A^D  = \mathop  \oplus \nolimits_{D' < D} R\left[ {D'} \right]$ or II.  $a \cdot R\left[ H \right] \subseteq D$ .   If I. holds, $a \cdot _{(R,s)} C_{(L,t)}^{(D,\lambda )}  \subseteq \hat A^D  \subseteq \hat A^{(D,\lambda )} $ and we can take the coefficients $r_L $ required in (i) to be all $0$.  So assume II. holds.  Take an element $x \in H$ and let $c = \pi (a,x)$.  Since $ax \in D = D_x $
, compatibility gives $\pi \left( {a,y} \right) = c$ for any $y \in H$.  But then, by linearity, $a \circ _{(R,s)} C_{(L,t)}^{(D,\lambda )}  = c \cdot a \cdot _{(R,s)} C_{(L,t)}^{(D,\lambda )} $.  We can then take the coefficients $r_L $ required in (i) to be just $c$ times the corresponding coefficients in the cell algebra $R[M]$.  

The proof of condition (ii) is parallel.  
\end{proof}

Let $M$ be a finite monoid satisfying the $R$-C.A. condition and place the standard cell algebra structure on $M$.  Suppose $R_i \,,\,L_j $ are matched classes in a $\mathcal{D}$-class $D$ of $M$.  Then by definition there exist $x \in L_j \,,\,y \in R_i $ such that $xy \in D$.  If $\pi $ is a compatible twisting on $M$ define $c(j,i) = \pi \left( {x,y} \right)$.  Then $\pi \left( {a,b} \right) = c(j,i)$ for any $a \in L_j \,,\,b \in R_i $, so for any $X \in R[L_j ]\,,\,Y \in R[R_i ]$ we have $X \circ Y = c(j,i)XY$.  It follows that for $X \in \left( {C_R ^{(D,\lambda )} } \right)_j \,,\,Y \in \left( {_L C^{(D,\lambda )} } \right)_i $ we have $\left\langle {X,Y} \right\rangle ^\pi   = c(j,i)\left\langle {X,Y} \right\rangle $ where $\left\langle {X,Y} \right\rangle ^\pi  $ is the bracket in the cell algebra $R^\pi  \left[ M \right]$.  

\begin{definition}   \label{d5.2}
  A twisting $\pi $ on a monoid $M$ is \emph{strongly compatible} if for all $a,x \in M$
\begin{enumerate} 
\item  If  $ax \mathcal{D} x$ , then $\pi \left( {a,y} \right) = \pi \left( {a,x} \right) \ne 0$
 whenever $y \in H_x $,
\item If  $xa \mathcal{D} x$ , then $\pi \left( {y,a} \right) = \pi \left( {x,a} \right) \ne 0$
 whenever $y \in H_x $.
\end{enumerate}
\end{definition}

So for a strongly compatible twisting we have $c(j,i) \ne 0$ whenever $R_i \,,\,L_j $ are matched classes.  But $c(j,i) \ne 0$ in the domain $R$ yields $\left\langle {X,Y} \right\rangle ^\pi   = 0 \Leftrightarrow \left\langle {X,Y} \right\rangle  = 0$.  Using this observation, it is easy to modify the proofs to obtain the following generalizations to twisted monoid algebras of the results in section 4.

In the following, assume $R$ is a domain, $M$ a finite monoid satisfying the $R$-C.A. condition, and $\pi $ a strongly compatible twisting from $M$ to $\pi $.  Put the standard cell algebra structures on $R[M]$ and $R^\pi[M]$ as given by theorems \ref{t3.1} and \ref{t5.1}.  Let $\Lambda ,\Lambda _0 ,\left\langle { - , - } \right\rangle  , \rad$, etc., refer to the cell algebra $R[M]$ and $\Lambda ^\pi  ,\Lambda ^\pi  _0 ,\left\langle { - , - } \right\rangle ^\pi \rad^\pi$, etc., refer to the cell algebra $R^\pi  (M)$.
Recall that $\Lambda ^\pi = \Lambda $. 
  
\begin{proposition}  \label{p5.1}
For any $\lambda  \in \Lambda _D \,,\,D \in \mathbb{D}$:
\begin{enumerate}
\item  If $\rad_D \left( {\,_L C^\lambda  } \right) \ne 0$, then $\rad^\pi \left( {\,_L C^{(D,\lambda )} } \right) \ne 0$,
\item  If $\rad_D \left( {D_R ^\lambda  } \right) \ne 0$, then $\rad^\pi \left( {D_R ^{(D,\lambda )} } \right) \ne 0$.
\end{enumerate} 
\end{proposition}

\begin{proposition}  \label{p5.2}
For any $\mathcal{D}$-class $D$,
\begin{enumerate}  [\upshape(a)]
\item   If $D^2  \subseteq \hat A^D $, then $(D,\lambda ) \notin \Lambda^\pi _0 $ for any $\lambda  \in \Lambda _D $
\item   If $D^2  \not\subset \hat A^D $, then $(D,\lambda ) \in \Lambda^\pi _0  \Leftrightarrow \lambda  \in \left( {\Lambda _D } \right)_0 $.
\end{enumerate}
\end{proposition}

\begin{proposition}  \label{p5.3}
 Let $R = k$ be a field. Assume that for every class $D \in \mathbb{D}$ , $D^2  \not\subset \hat A^D $ and $k[G_D ]$ is a cell algebra with $\left( {\Lambda _D } \right)_0  = \Lambda _D $.  Then $A^\pi = k^\pi[M]$ is quasi-hereditary.
\end{proposition}

\begin{proposition}  \label{p5.4}
  Let $R =k $ be a field.  Assume that $M$ is regular and that for every class $D \in \mathbb{D}$ , $k[G_D ]$ is a cell algebra with $\left( {\Lambda _D } \right)_0  = \Lambda _D $.  Then $A^\pi = k^\pi[M]$ is quasi-hereditary.
\end{proposition}

\begin{proposition}  \label{p5.5}
If $R = k$ is a field and $k^\pi[M]$ is semi-simple, then $k[G_D]$ is semi-simple for every $D \in \mathbb{D} $.
\end{proposition}

\begin{proposition}  \label{p5.6}
Let $R = k$ be a field.  Assume that for every $\mathcal{D}$-class $D$
\begin{enumerate}
\item   the cell algebra $k[G_D ]$ is a semi-simple and
\item   the bijection condition is satisfied for $D$.
\end{enumerate}
Then the cell algebra $k^\pi[M]$ is semi-simple.
\end{proposition}

\begin{proposition}  \label{p5.7}
  Let $R = k$ be a field.  Assume the monoid $M$ is also an inverse semi-group.   Then $k^\pi[M]$ is semi-simple if and only if $k[G_D]$ is semi-simple for every $\mathcal{D}$-class $D$. 
\end{proposition}

The various examples such as Brauer algebras, Temperly-Lieb algebras, and other parition algebras which were studied and shown to be cellular in \cite{Wil} and \cite{GX} are all twisted monoid algebras with a compatible twisting on a monoid satisfying the $R$-C.A. condition.  Thus they can be seen to be cell algebras by Theorem \ref{t5.1} without constructing the anti-isomorphism needed for the cellular structure. Related algebras which lack the anti-isomorphism $ * $, and hence are not cellular, could also be shown to be cell algebras by Theorem \ref{t5.1}.  We note again that questions such as whether an algebra is quasi-hereditary or semi-simple are not much harder to answer for cell algebras than for cellular algebras.

\end{document}